\documentclass[12pt]{amsart}
\usepackage{verbatim, amsmath,amsthm,amssymb, amscd, amsfonts}
\newtheorem{theorem}{Theorem}[section]
\newtheorem{lemma}[theorem]{Lemma}
\newtheorem{proposition}[theorem]{Proposition}
\theoremstyle{definition}
\newtheorem{definition}[theorem]{Definition}
\newtheorem{example}[theorem]{Example}
\newtheorem{rema}[theorem]{Remark}
\newtheorem{coro}[theorem]{Corollary}

\begin{document}
\newcommand{\nc}{\newcommand}
\nc{\rnc}{\renewcommand} \nc{\nt}{\newtheorem}


\nc{\TitleAuthor}[2]{\nc{\Tt}{#1}%
    \nc{\At}{#2}%
    \maketitle%
}


\nc{\Hom}{\operatorname{Hom}} \nc{\Mor}{\operatorname{Mor}} \nc{\Aut}{\operatorname{Aut}}
\nc{\Ann}{\operatorname{Ann}} \nc{\Ker}{\operatorname{Ker}} \nc{\Trace}{\operatorname{Trace}}
\nc{\Char}{\operatorname{Char}} \nc{\Mod}{\operatorname{Mod}} \nc{\End}{\operatorname{End}}
\nc{\Spec}{\operatorname{Spec}} \nc{\Span}{\operatorname{Span}} \nc{\sgn}{\operatorname{sgn}}
\nc{\Id}{\operatorname{Id}} \nc{\Com}{\operatorname{Com}}

\nc{\nd}{\mbox{$\not|$}} 
\nc{\nci}{\mbox{$\not\subseteq$}}
\nc{\scontainin}{\mbox{$\mbox{}\subseteq\hspace{-1.5ex}\raisebox{-.5ex}{$_\prime$}\hspace*{1.5ex}$}}


\nc{\R}{{\sf R\hspace*{-0.9ex}\rule{0.15ex}%
    {1.5ex}\hspace*{0.9ex}}}
\nc{\N}{{\sf N\hspace*{-1.0ex}\rule{0.15ex}%
    {1.3ex}\hspace*{1.0ex}}}
\nc{\Q}{{\sf Q\hspace*{-1.1ex}\rule{0.15ex}%
       {1.5ex}\hspace*{1.1ex}}}
\nc{\C}{{\sf C\hspace*{-0.9ex}\rule{0.15ex}%
    {1.3ex}\hspace*{0.9ex}}}
\nc{\Z}{\mbox{${\sf Z}\!\!{\sf Z}$}}


\newcommand{\gd}{\delta}
\newcommand{\sub}{\subset}
\newcommand{\cntd}{\subseteq}
\newcommand{\go}{\omega}
\newcommand{\Pa}{P_{a^\nu,1}(U)}
\newcommand{\fx}{f(x)}
\newcommand{\fy}{f(y)}
\newcommand{\gD}{\Delta}
\newcommand{\gl}{\lambda}
\newcommand{\half}{\frac{1}{2}}
\newcommand{\ga}{\alpha}
\newcommand{\gb}{\beta}
\newcommand{\gga}{\gamma}
\newcommand{\ul}{\underline}
\newcommand{\ol}{\overline}
\newcommand{\Lrraro}{\Longrightarrow}
\newcommand{\equi}{\Longleftrightarrow}
\newcommand{\gt}{\theta}
\newcommand{\op}{\oplus}
\newcommand{\Op}{\bigoplus}
\newcommand{\CR}{{\cal R}}
\newcommand{\tr}{\bigtriangleup}
\newcommand{\grr}{\omega_1}
\newcommand{\ben}{\begin{enumerate}}
\newcommand{\een}{\end{enumerate}}
\newcommand{\ndiv}{\not\mid}
\newcommand{\bab}{\bowtie}
\newcommand{\hal}{\leftharpoonup}
\newcommand{\har}{\rightharpoonup}
\newcommand{\ot}{\otimes}
\newcommand{\OT}{\bigotimes}
\newcommand{\bwe}{\bigwedge}
\newcommand{\gep}{\varepsilon}
\newcommand{\gs}{\sigma}
\newcommand{\OO}{_{(1)}}
\newcommand{\TT}{_{(2)}}
\newcommand{\FF}{_{(3)}}
\newcommand{\minus}{^{-1}}
\newcommand{\CV}{\cal V}
\newcommand{\CVs}{\cal{V}_s}
\newcommand{\slp}{U_q(sl_2)'}
\newcommand{\olp}{O_q(SL_2)'}
\newcommand{\slq}{U_q(sl_n)}
\newcommand{\olq}{O_q(SL_n)}
\newcommand{\un}{U_q(sl_n)'}
\newcommand{\on}{O_q(SL_n)'}
\newcommand{\ct}{\centerline}
\newcommand{\bs}{\bigskip}
\newcommand{\qua}{\rm quasitriangular}
\newcommand{\ms}{\medskip}
\newcommand{\noin}{\noindent}
\newcommand{\raro}{\rightarrow}
\newcommand{\alg}{{\rm Alg}}
\newcommand{\rcom}{{\cal M}^H}
\newcommand{\lcom}{\,^H{\cal M}}
\newcommand{\rmod}{\,_R{\cal M}}
\newcommand{\qtilde}{{\tilde Q^n_{\gs}}}
\nc{\e}{\overline{E}} \nc{\K}{\overline{K}} \nc{\gL}{\Lambda}
\newcommand{\tie}{\bowtie}
\newcommand {\h}{\widehat}
\newcommand {\tl}{\tilde}
\newcommand{\tri}{\triangleright}
\nc{\ad}{_{\dot{ad}}}
\nc{\coad}{_{\dot{{\rm coad}}}}
\nc{\ov}{\overline}

\title{Probabilistically nilpotent Hopf algebras }
\author{Miriam Cohen}
\address {Department of Mathematics\\
Ben Gurion University, Beer Sheva, Israel}
\email {mia@math.bgu.ac.il}
\author{Sara Westreich}
\address{Department of Management\\
Bar-Ilan University,  Ramat-Gan, Israel}
\email{swestric@biu.ac.il}
\thanks {This research was supported by the ISRAEL SCIENCE FOUNDATION, 170-12.}


\date{Sep-2013}
\begin{abstract}
In this paper we investigate nilpotenct and probabilistically  nilpotent Hopf algebras.  We define nilpotency via a descending chain of commutators and give a criterion for nilpotency via a family of central invertible elements. These elements can be obtained from a {\it commutator matrix} $A$ which depends only on the Grothendieck ring of $H.$ When $H$ is almost cocommutative we introduce a probabilistic method. We prove that every semisimple quasitriangular Hopf algebra is probabilistically nilpotent. In a sense we thereby answer the title of our  paper {\it Are we counting or measuring anything?}   by {\it Yes we are}.
\end{abstract}
\maketitle

\section*{Introduction}
Following our previous generalizations of classical and recent ideas about finite groups, this paper deals with nilpotent and probabilistically  nilpotent Hopf algebras. 
 
Nilpotency of  a group is defined via an ascending chain of normal  subgroups corresponding to centers of group quotients. We generalize this ideas to give an intrinsic definition of nilpotency for semisimple Hopf algebras $H$ over a field of characteristic $0.$ The role of the center of a group is played by the Hopf center\cite{an} and the role of normal subgroups  is played by normal left coideal subalgebras of $H.$ Commutator subgroups  provide an alternative  way of defining nilpotency  for groups. Here again our previous work on commutators \cite{cwcom} enables us to generalize these to semisimple  Hopf algebras.

In order to give an easy criterion for nilpotency we define a recursive set of central invertible elements, $\{\gga_m\},$ which are obtained by applying a {\it Commutator operator} on the center of $H$ and  can be realized via a corresponding {\it Commutator matrix}. This matrix,  motivated by \cite{av}, depends only on the Grothendieck ring of $H.$ We show that under a minor assumption on it, $H$ is nilpotent  if and only if $\gga_m=1$ for some $m.$ This is in fact an extension of our previous result in \cite{cwcom}, where we proved, even without the mentioned minor assumption, that $H$ is a commutative algebra (in particular nilpotent) if and only if $\gga_1=1.$

Probabilistic methods for group theory were introduced in the early 1960’s by Erdos and Renyi and have been since applied  with a great deal of success. Many of the methods involve character theory and conjugacy classes (see e.g  \cite{aner}). 

This paper is a first attempt to adopt probabilistic methods  for Hopf algebras.
 While finite groups are naturally probabilistically nilpotent, we prove here that the same is true for semisimple quasitriangular Hopf algebras. 
We show  how some counting functions on groups and their generalization to Hopf algebras as in \cite{cwcom}, can now be realized as {\it distribution functions} and thus can be considered as {\it measuring }. We thereby answer the title of our  paper\cite{cwcom} {\it Are we counting or measuring anything?}   by {\it Yes we are}.

Throughout this paper we assume $H$ is a semisimple Hopf algebra over a field of characteristic $0$ and $\gL$ is its idempotent integral. We denote the character algebra of $H$ by $R(H)$ and the center of $H$ by $Z(H).$ The paper is organized as follows:

\bigskip In $\S 1$ we give the following intrinsic definition of nilpotency for Hopf algebras. For any normal left coideal subalgebra $N$ of $H$ let $H//N$ denote the Hopf quotient $H/H(N\cap\ker\gep).$
Let $\tl{Z}(H)$ denote the Hopf center of $H$ and define a series of normal left coideal subalgebras of H as follows.
Set:
$$H_0=H, \;\pi_0=\Id,\;\,Z_0=k,$$ and set by induction for $i>0,$
$$H_{i}=H_{i-1}//\tl{Z}(H_{i-1}),\;\pi_i:H\longrightarrow H_i,\; Z_i=H^{co\pi_i}.$$

\medskip\noin{\bf Definition \ref{nil}}: A semisimple Hopf algebra $H$ is {\bf nilpotent}  if the ascending central series
$$k\subseteq Z_1\subseteq Z_2\subseteq \cdots$$
satisfies $Z_m=H$ for some $m\ge 1.$ The smallest  such $m$ is called the index of nilpotency of $H.$

\medskip
A categorical definition of nilpotency was given in \cite{gn}. We show how it coincides with our intrinsic definition.

\medskip
For the commutator approach, define as in \cite{cwcom} the generalized commutator $\{a,b\},$ 
$$\{a,b\}=\sum a_1b_1Sa_2Sb_2$$
   for any $a,b\in H.$
	Define an ascending chain of iterated commutators  by:
$$N_0=H,\;N_1= [\{H,\gL\}],\;\ldots,N_t=[\{N_{t-1},\gL\}]$$
where $[S]$ is the normal left coideal subalgebra generated by any subset $S$ of $H.$  
We prove: 

\medskip\noin{\bf Theorem \ref{equal}}
Let $H$ be a semisimple Hopf algebra over an algebraically closed field of characteristic $0.$  Then
$$N_t=k1\Longleftrightarrow Z_t=H.$$
In this case $N_{t-i}\subseteq Z_{i}$ for all $0\le i\le t.$

\bigskip
In $\S 2$ we define the {\it Commutator operator} $T:Z(H)\longrightarrow Z(H)$ by:
$$T(z)=\{z,\gL\}.$$

Let $\{E_i\},\;\{\chi_i\},\;0\le i\le n-1,$ be the full set of central idempotents of $H$ and  the corresponding set of  irreducible characters of degree $d_i.$ 
We prove:

\medskip\noin{\bf Proposition \ref{matrixA}}:(i) The matrix of $T$ with respect to the basis $\{\frac{E_0}{d_0^2},\dots, \frac{E_{n-1}}{d_{n-1}^2}\}$ is  
$A,$ where
$$
A_{ij}=\frac{\langle\chi_is(\chi_i)s(\chi_j),\gL\rangle}{d_j},\quad 0\le i,j\le n-1.$$
\medskip\noin (ii) $A$ has non-negative rational entries and the first column of $A$ has all entries equal $1.$ 

\medskip\noin (iii) The first row of $A^m$ is $(1,0,\dots,0)$ for all $m\ge 0.$ 

\medskip\noin (iv)  In the first column of $A^m$ we have: 
$$(A^m)_{i0}=\sum_j(A^{m-1})_{ij},$$
For all $m> 0,\;0\le i\le n-1.$ 
 In particular, The first column of $A^m$ consists of positive rational numbers.

\medskip We refer to $A$ as the {\it Commutator matrix} of $H.$ Since this matrix depends only on the Grothendieck ring of $H,$ it follows that the commutator matrix is a categorical invariant.

We define the following family of central iterated commutators which play a key role in the sequel. 
$$\gamma_0=\gL,\,\gamma_1=\{\gL,\gL\},\,\dots,\,\gamma_m=\{\gamma_{m-1},\gL\}.$$

We prove:

\medskip\noin{\bf Theorem \ref{connection}}: Let $H$ be a semisimple Hopf algebra over an algebraically closed field of characteristic $0,$ and let $N_t,\,\gga_t$ be defined as above. Assume $\chi_is(\chi_i)\in Z(R(H))$ for each irreducible character $\chi_i.$ Then $H$ is nilpotent if and only if $\gga_m=1$ for some $m\in\Z^+.$  
Its index of nilpotency is  the least integer $m$ so that $\gga_m=1.$

\medskip A significant corollary is the following:

\medskip\noin{\bf Theorem \ref{atothm}}
Let $H$ be a semisimple Hopf algebra over an algebraically closed field of characteristic $0$ and assume $\chi_is(\chi_i)\in Z(R(H))$ for each irreducible character $\chi_i.$ Then $H$ is nilpotent  if and only if its commutator matrix $A$ has eigenvalues $\{1,0\}$ and the algebraic multiplicity of $1$ is $1.$

\bigskip In $\S 3$ we introduce a probabilistic method for semisimple Hopf algebras such that $R(H)$ is  commutative. For $0\le i\le n-1,$ let  $\{F_i\}$ be a full set of primitive idempotents of $R(H).$  In this setup we have defined in \cite{cwmia} Hopf algebraic analogues of conjugacy classes ${\mathfrak C}_i$, class sums $C_i$ and normalized class sums $\eta_i.$ The set $\{\eta_i\}$ form a basis for $Z(H).$

We call an element $z\in Z(H)$  a {\bf central distribution element} if 
$$z=\sum\ga_i\eta_i,\;\ga_i\in \R^+\cup\{0\},\quad \sum_i\ga_i=1.$$  
The idempotent integral $\gL$ and all $\eta_i$ are such elements. The central distribution element $z$ defines a  distribution $X_z$ on $H$ by letting:
$${\rm Prob}(X_z=C_i)=\ga_i$$
The corresponding distribution function $f_z$ is given by:
$$f_z(C_i)={\rm Prob}(X_z=C_i)=\ga_i.$$

\medskip
Counting functions for groups give rise to  distribution functions on the group algebras after dividing by an appropriate power of $|G|.$ 
For example, Frobenius proved  that the function on a finite group $G$ that counts the number of ways an element of $G$ can be realized as a commutator  is given by:
$$f_{rob}=\sum_i \frac{|G|}{d_i}\chi_i.$$
Then  $f= \frac{1}{|G|^2}f_{rob}$ is a distribution function for $kG$ with a corresponding central distribution element $z=\frac{1}{|G|^2}\sum_{a,b\in G}aba\minus b\minus.$ The Hopf algebra analogue of $f$ is $\frac{1}{d}\sum_i \frac{1}{d_i}\chi_i.$ It is the distribution function corresponding to the central distribution element $z=\gga_1.$

\medskip When $H$ is quasitriangular then our {\it commutator map} $T$ maps central distribution elements to central distribution elements. In particular, all $\gamma_m$ are central distribution elements. As a consequence we prove:

\medskip\noin{\bf Proposition \ref{converge}}: Let $z\in Z(H)$ be a central distribution element and let $T(z)=\{z,\gL\}.$ Then $${\rm Prob}(T^m(z)=1)\longrightarrow 1\quad\text{as}\quad m\rightarrow \infty.$$

\medskip These give rise to the following definition:

\medskip\noin{\bf Definition \ref{pnil}}: A semisimple Hopf algebra is {\bf Probabilistically nilpotent} if 
$${\rm Prob}(\gga_m=1)\longrightarrow 1\quad\text{as}\quad m\rightarrow \infty.$$

The main result here is:

\medskip\noin{\bf Theorem \ref{main3}}:
Let  $H$ be a semisimple quasitriangular Hopf algebra over an algebraically closed field of characteristic $0.$ Then $H$ is probabilistically nilpotent. 

\medskip Note that even if $H$ is far from being nilpotent, for example if its Hopf center is $k,$ $H$ is still probabilistically nilpotent. This fact is demonstrated when  
we characterize the eigenvalues of the matrix $A$ over $\C.$

\medskip\noin{\bf Theorem \ref{eigen}}: Let $H$ be a semisimple quasitriangular Hopf algebra over $\C.$ Then the commutator matrix $A$  has $1$ as an eigenvalue with corresponding $1$-dimensional eigenspace. 
 All other eigenvalues $c$ satisfy $|c|<1.$

\section{Upper and lower central series}
Throughout this paper,  $H$ is a semisimple Hopf algebra over an algebraically closed field $k$ of characteristic $0.$ We denote by $S$ and $s$ the antipodes of $H$ and $H^*$ respectively. Denote by  $Z(H)$ the center of $H.$

The Hopf algebra $H^*$ becomes a right and left $H$-module by the {\it hit} actions $\leftharpoonup$ and $\rightharpoonup$ defined for all $a\in H,\,p\in H^*,$
$$\langle  p\leftharpoonup a,a'\rangle  =\langle  p,aa'\rangle  \qquad \langle  a\rightharpoonup p,a'\rangle  =\langle  p,a'a\rangle  $$
$H$ becomes a left and right  $H^*$-module analogously.

Denote by $_{\dot{ad}} $ the left adjoint action of $H$ on itself, that is, for all $a,h\in H,$
$$h_{\dot{ad}}  a=\sum h_1aS(h_2)$$

A left coideal subalgebra of $H$ is called {\it normal} if it is stable under the left adjoint action of $H.$

Let  $\{V_0,\dots V_{n-1}\}$  be a complete set of non-isomorphic irreducible $H$-modules. Let $\{E_0,\dots E_{n-1}\}$ and ${\rm Irr}(H)=\{\chi_0,\dots,\chi_{n-1}\}$ be the associated central primitive idempotents and  irreducible characters of $H$ respectively, where $E_0=\gL,$ the idempotent integral of $H$ and $\chi_0=\gep.$  
Let $\dim V_i=d_i=\langle  \chi_i,1\rangle.$  Then
$$\gl=\chi_H=\sum_{i=0}^{n-1}d_i\chi_i$$
is an integral for $H^*$ satisfying $\langle\gl,\gL\rangle=1.$

\medskip
For any normal left coideal subalgebra $N$ of $H$ we denote by $H//N$ the Hopf quotient $\ol{H}=H/HN^+,$ where $N^+=N\cap\ker\gep.$  
Recall for any Hopf surjection $\pi:H\rightarrow\ol{H}$ we define,
$$H^{co\pi}=\{h\in H|\sum h_1\ot\pi(h_2)=h\ot\pi(1)\}.$$
Then by \cite{ta}
$$H//H^{co\pi}\cong\ol{H}.$$

Related concepts are invariants. For any subalgebra $T$ of $H^*,$ we denote by $H^{T}$ the set of $T$-invariants of $H$ under the left {\it hit} action.
That is,
\begin{equation}\label{hb}H^{T}=\{h\in H\,|b\rightharpoonup h=\langle  b,1\rangle  h,\;\forall b\in T\}
\end{equation}
It was shown in \cite[Prop.2.2]{cwromania} that $B$ is a Hopf subalgebra of $H^*$ if and only if $N=H^B$ is a normal left coideal subalgebra of $H.$ In this case one also has $B=(H^*)^N.$ 

Fix $\pi:H\rightarrow\ol{H},\;B=\ol{H}^*\subset H^*$ and $N$ a left coideal subalgebra then summing up: 
\begin{equation}\label{recall} N=H^{co\pi}\Leftrightarrow\ol{H}= B^*\cong H//N\Leftrightarrow \ol{H}^*=B=(H^*)^N\Leftrightarrow N=H^B.\end{equation}

\medskip
For any $H$-representation $V$  the left kernel ${\rm LKer}_{V}$ is defined as:
\begin{equation}\label{lker0}{\rm LKer}_{V}=\{h\in H\,|\,\sum h_1\ot h_2 \cdot v=h\ot v\quad\forall v\in V\}.\end{equation}
Considering $\ol{H}$ as an $H$-representation it follows  (see e.g. \cite{bu}) that 
$$H^{co\pi}= {\rm LKer}_{\ol{H}}.$$


\bigskip

The Hopf center of a Hopf algebra was introduced and described in \cite{an}. It is the maximal Hopf subalgebra of $H$ contained in $Z(H).$ A categorical version of it is used  in \cite{gn} to define upper central series for fusion categories. In what follows we present an intrinsic approach for semisimple Hopf algebras. One of the advantages of this approach is that it gives rise to a descending chain of commutators of $H$ as well. 

\medskip

Let $H^*_{ad}$ be the Hopf subalgebra of $H^*$ generated by 
$$\chi_{ad}=\sum_j\chi_js(\chi_j),\;\chi_j\in {\rm Irr}(H).$$ 
Equivalently, the Hopf algebra generated by the irreducible constituents of $\chi_is(\chi_i),\,0\le i\le n-1.$ 

Since $\chi_{ad}$ and all its powers are central in $R(H)$ we have for each $l,$
$\chi_i\chi_{ad}^ls(\chi_i)=\chi_{ad}^l\chi_is(\chi_i)\in H^*_{ad}.$ 
Hence $D_iD_{\chi_{ad}^l}s(D_i)\subset H^*_{ad},$ where 
$D_i$ is the simple coalgebra generated by $\chi_i.$ It follows
 that $H^*_{ad}$ is a normal Hopf subalgebra of $H^*.$ 
Define:
\begin{equation}\label{center}\tl{Z}(H)=H^{H^*_{ad}}\subset H.\end{equation}
Take $V=(H,\,\ad)$ then By \cite[Th.2.8]{cwromania}, $\tl{Z}(H)= {\rm LKer}_V.$ Explicitly, 
\begin{equation}\label{lker}\tl{Z}(H)=\{h\in H\,|\,\sum h_1\ot h_2 xSh_3=h\ot x\quad\forall x\in H\}.\end{equation}

\medskip The following proposition is a variation of \cite{an}.
\begin{proposition}\label{HupperB}
Let $H$ be a semisimple Hopf algebra over an algebraically closed field of characteristic $0.$ Then $\tl{Z}(H)$ is the Hopf center of $H.$ Moreover, $\tl{Z}(H)$  contains every left (right) coideal of $H$ contained in $Z(H).$ 
\end{proposition}
\begin{proof}

Observe first that every left coideal $N$ contained in $Z(H)$ is contained in $\tl{Z}(H).$  Indeed, if $h\in N\subset Z(H)$ then   sp$_k\{h_3\}\subset N\subset Z(H)$ implying that $h$  satisfies the right hand side of \eqref{lker}.

We wish to show now that $\tl{Z}(H)$ is  contained in $Z(H).$ 
Let $h\in \tl{Z}(H).$ Applying $\mu\circ(S\ot\Id)$ to the right hand side of \eqref{lker} yields $xSh=Shx$ for all $x\in H,$ hence $Sh$ and thus $h\in Z(H).$ 

 Since $H^*_{ad}$ is normal in $H^*$ it follows from \cite[Prop.2.2]{cwromania} that $\tl{Z}(H)=H^{H^*_{ad}}$ is a Hopf subalgebra of $H.$ 
\end{proof}

\medskip
In what follows we generalize directly the notion of upper central series from finite groups to semisimple Hopf algebras. 
We define an ascending central series of normal left coideal subalgebras $\{Z_n\}.$ As for groups this series corresponds to Hopf centers of Hopf quotients. Set:
$$H_0=H, \;\pi_0=\Id,\;\,Z_0=k,$$ and set by induction for $i>0,$
\begin{equation}\label{z}H_{i}=H_{i-1}//\tl{Z}(H_{i-1}),\;\pi_i:H\longrightarrow H_i,\; Z_i=H^{co\pi_i}.\end{equation}
Then
\begin{equation}\label{quo}Z_1=\tl{Z}(H),\quad H_{i}=H//Z_{i}.\end{equation}
\begin{definition}\label{nil} A semisimple Hopf algebra $H$ is {\bf nilpotent}  if the ascending central series
$$k\subseteq Z_1\subseteq Z_2\subseteq \cdots$$
satisfies $Z_m=H$ for some $m\ge 1.$ The smallest  such $m$ is called the index of nilpotency of $H.$
\end{definition}

An alternative way to realize $\{Z_n\}$ arises from \cite[Def. 4.1]{gn}. One has a descending chain of Hopf subalgebras of $H^*$ defined as follows:
\begin{equation}\label{b}B_0=H^*,\; B_1=H^*_{ad},\;\ldots, B_{i+1}=B_{i_{ad}}.\end{equation}
If $B_m=k$ then $H$ is nilpotent. 
The following lemma enables us to connect the two definitions of nilpotency:
\begin{lemma}\label{zn}
Let $Z_n$ and $B_n$ be defined as in \eqref{z} and \eqref{b} respectively.  Then  for $n\ge 0$ we have $$Z_n=H^{B_n}.$$ 

In particular,  the intrinsic definition and the categotical definition of nilpotency coincide. 
\end{lemma}
\begin{proof}
By \eqref{z}, $H_1=H//Z_1,$   hence by \eqref{recall} and  \eqref{center} $Z_1=H^{B_1}.$    Assume by induction $Z_i=H^{B_i},$ 
then by \eqref{quo} and \eqref{recall},
$$H_{i}=H//Z_i\cong B_i^*.$$
It follows from \eqref{center} and \eqref{b} that
$$\tl{Z}(H_{i})=H_{i}^{B_{i_{ad}}}=H_{i}^{B_{i+1}}$$
hence by \eqref{z} and \eqref{recall},
$$H_{i+1}=H_i//\tl{Z}(H_i)\cong B_{i+1}^*.$$
But by \eqref{quo} $H_{i+1}\cong H//Z_{i+1},$ hence  $H//Z_{i+1}\cong B_{i+1}^*$ by the formula in the line above. This implies by \eqref{recall}  that:
$$Z_{i+1}=H^{B_{i+1}}.$$
\end{proof}

\bigskip
For groups there is  a related descending chain of subgroups arising from commutators. In what follows we generalize this idea to Hopf  algebras. As in \cite{cwcom}, for $a,b\in H,$ set their commutator 
$$\{a,b\}=\sum a_1b_1Sa_2Sb_2.$$
 
Commutators and Hopf centers are related in the following way.
\begin{proposition}\label{0}
Let $N$ be a normal left coideal subalgebra of $H.$ Then 
for all $t\ge 0,$
$$\{N,\gL\}\subseteq Z_t\Longleftrightarrow N\subseteq Z_{t+1}.$$  
\end{proposition}
\begin{proof}
Assume first $\{N,\gL\}=k.$ We show that if $x\in N$ then $S(x)\in Z(H).$ 
Indeed, $$\gL\ad Sx=\sum Sx_1x_2\gL_1 Sx_3 S\gL_2=\sum Sx_1\{x_2,\gL\}=Sx.$$ 
The last equality follows from the assumption since $N$ is a left coideal.
Thus $x\in Z(H)$ as well.  Since $Z_1$ is the Hopf center of $H$ it  follows from Lemma \ref{HupperB} that  $N\subseteq Z_1.$

Conversely, if $N\subseteq Z_{1}$ then in particular $N\subseteq Z(H).$ Hence for all $x\in N,\;\sum\gL_1x_1S\gL_2Sx_2=\langle \gep,x\rangle.$ 

We continue by induction. Note $\pi(\gL)$ is an idempotent integral for any homomorphic image $\pi(H)$ of $H.$  Hence 
$$\{N,\gL\}\subseteq Z_t\Longrightarrow  \{\pi_t(N),\pi_t(\gL)\}=\pi_t(\{N,\gL\})=k.$$
But $\pi_t(N)$ is a normal left coideal subalgebra of $H_t,$ hence by the first step of the induction proved above we obtain 
$\pi_t(N)\subset \tl{Z}(H_t).$
Now, by definition, $H_{t+1}\cong H_t//\tl{Z}(H_t),$ hence $N$ is mapped under $\pi_{t+1}$ into $k.$ Since $N$ is a normal left coideal subalgebra it follows that $$N\subset H^{co\pi_{t+1}}=Z_{t+1}.$$

Conversely, assume $N\subset Z_{t+1}.$  Then $\pi_{t+1}(N)=k$ and since $\pi_t(N)$ is a left coideal it follows that $\pi_t(N)\subset \tl{Z}(H_t).$   Let $m=\sum n_1\gL_1Sn_2S\gL_2\in \{N,\gL\},$ then 
$$\sum m_1\ot \pi_t(m_2)=\sum n_1\gL_1Sn_4S\gL_4\ot\pi_t(n_2)\pi_t(\gL_2)\pi_t(Sn_3)\pi_t(S\gL_3)=m\ot\pi_t(1).$$
The last equality holds since $\pi_t(N)\subset \tl{Z}(H_t).$ This implies that the subcoalgebra generated by $\pi_t(N)$ is contained in $\tl{Z}(H_t)\subset Z(H_t).$  Thus $m\in Z_t.$
\end{proof} 

\medskip
For a set $S$ let $[S]$ denote the normal left coideal subalgebra generated by $S.$  Define a descending chain of iterated commutators for $H$ as follows:
\begin{equation}\label{N}N_0=H,\;N_1= [\{H,\gL\}],\;\ldots,N_t=[\{N_{t-1},\gL\}]\end{equation}
By induction, if $N_t\subseteq N_{t-1}$ then $\{N_{t},\gL\}\subseteq \{N_{t-1},\gL\}$ hence $$N_{t+1}=[\{N_{t},\gL\}]\subseteq [\{N_{t-1},\gL\}]=N_t.$$

\begin{rema}\label{groups2} 
Let $H=kG,\;G$ a finite group.  Then $\tl{Z}(H)=kZ_G,$ where $Z_G$ is the center of the group $G.$ This follows since every Hopf subalgebra of $kG$ has the form $kK$ where $K$ is a subgroup of $G.$ 
For any Hopf quotient we have $kG//kK\cong k(G/K).$ Moreover, for any Hopf surjection $\pi$  we have $H^{co\pi}=k\pi\minus(1).$ These observations imply that an ascending central series for $G$ gives rise to an ascending central series for $kG$ and vice verse. 

If we denote by induction $G_1=[G,G]$ and $G_{t+1}=[G_t,G],$ then   $N_t=kG_t.$  
\end{rema}
We can prove now the main result of this section.
\begin{theorem}\label{equal}
Let $H$ be a semisimple Hopf algebra over an algebraically closed field of characteristic $0,$ let $Z_t$ be defined as in  \eqref{z} and $N_t$ as defined in \eqref{N}. Then
$$N_t=k1\Longleftrightarrow Z_t=H.$$
In this case $N_{t-i}\subseteq Z_{i}$ for all $0\le i\le t.$
\end{theorem}
\begin{proof}
We claim first that for all $0\le i\le t,$
\begin{equation}\label{nit}N_{t-i}\subseteq Z_{i}\Longleftrightarrow N_{t-i-1}\subseteq Z_{i+1}.\end{equation}
Indeed, assume $N_{t-i}\subseteq Z_{i}.$ Since by definition $\{N_{t-i-1},\gL\}\subseteq N_{t-i},$  it follows in particular that  $\{N_{t-i-1},\gL\}\subseteq Z_i.$ By Proposition \ref{0}   $N_{t-i-1}\subseteq Z_{i+1}.$  Conversely, if $N_{t-i-1}\subseteq Z_{i+1}$ then  $\{N_{t-i-1},\gL\}\subset Z_i$ by Proposition \ref{0}. 
Since $N_{t-i}$ is generated as a normal left coideal by $\{N_{t-i-1},\gL\},$ and since $Z_i$ is a normal left coideal subalgebra containing $\{N_{t-i-1},\gL\},$ it follows that  $N_{t-i}\subseteq Z_{i}.$ This proves our claim.

Now, if $N_t=k,$ then by \eqref{nit}, $N_{t-1}\subset Z_1,$ and by induction on $i$ $N_{t-i}\subset Z_i.$  For $i=t$ we get $H=N_0\subset Z_t.$

Conversely, assume that $Z_t=H=N_0,$ then we use \eqref{nit} again to prove by induction on $j$ that $N_j\subset Z_{t-j}$ for all $j.$ When $j=t$ we obtain $N_t=k1.$ 
\end{proof}

\section{Iterated commutators}

Recall,  (see e.g \cite[Cor.4.6]{sc}):
 \begin{equation}\label{dual1}\langle  \chi_i,E_j\rangle  =\gd_{ij}d_j,\quad\chi_i\leftharpoonup E_j=\gd_{ij}\chi_i,\quad \gL\leftharpoonup s(\chi_j)=\frac{1}{d_j}E_j.\end{equation}
In particular, $\{\chi_i\},\,\{\frac{1}{d_j}E_j\}$ are dual bases of $R(H)$  and $Z(H)$ respectively.
Hence we have for each $z\in Z(H),\,p\in R(H)$
\begin{equation}\label{zzz}z=\sum_i \frac{1}{d_i}\left\langle \chi_i,z\right\rangle E_i\qquad p=\sum_i\frac{1}{d_i}\langle p,E_i\rangle\chi_i.\end{equation}
By \eqref{dual1} and \eqref{zzz} we have:
\begin{equation}\label{act}
\chi_i\leftharpoonup z = \frac{1}{d_i}\left\langle \chi_i,z\right\rangle\chi_i\end{equation}
for all $i.$
Since $\gL\ad H=Z(H)$ we have for all $h\in H,$
\begin{eqnarray*}\lefteqn{\langle\chi_i,\{h,\gL\}\rangle=}\\
&=&\sum \langle\chi_i,\gL_1h_1S\gL_2Sh_2\rangle=\sum \langle\chi_i\leftharpoonup\gL_1h_1S\gL_2,Sh_2\rangle=\frac{1}{d_i} \langle\chi_is(\chi_i),h\rangle.\end{eqnarray*}
where the last equality follows from \eqref{act}.
Thus
\begin{equation}\label{zi}\langle\chi_i,\{h,\gL\}\rangle=\frac{1}{d_i} \langle\chi_is(\chi_i),h\rangle.\end{equation}

Define an operator $T:Z(H)\longrightarrow Z(H)$ by
$$T(z)=\{z,\gL\}.$$
Indeed, by \cite[Lemma 2.3]{cwcom}, $T(z)\in Z(H)$ for all $z\in Z(H).$ Moreover, by definition of $N_m$ we have
\begin{equation}\label{tnm}
T^m(z)\in N_m\quad\forall z\in Z(H)\end{equation}
We have:
\begin{proposition}\label{matrixA}
(i) The matrix of $T$ with respect to the basis $\{\frac{E_0}{d_0^2},\dots, \frac{E_{n-1}}{d_{n-1}^2}\}$ is  
$A,$ where
\begin{equation}\label{A}
A_{ij}=\frac{\langle\chi_is(\chi_i)s(\chi_j),\gL\rangle}{d_j},\quad 0\le i,j\le n-1.\end{equation}

\medskip\noin (ii) $A$ has non-negative rational entries and the first column of $A$ has all entries equal $1.$ 

\medskip\noin (iii) The first row of $A^m$ is $(1,0,\dots,0)$ for all $m\ge 0.$ 

\medskip\noin (iv)  In the first column of $A^m$ we have: 
$$(A^m)_{i0}=\sum_j(A^{m-1})_{ij},$$
For all $m> 0,\;0\le i\le n-1.$ 
 In particular, The first column of  $A^m$ consists of positive rational numbers.
\end{proposition} 
\begin{proof}
(i) By dual bases of $Z(H)$ and $R(H)$ and by \eqref{zi} we have:
\begin{eqnarray*}\lefteqn {T(\frac{E_j}{d_j^2})=\{\frac{E_j}{d_j^2},\gL\}=}\\
&=&\sum_i \langle \{\frac{E_j}{d_j^2},\gL\},\chi_i\rangle\frac{E_i}{d_i}=\sum_i \langle \frac{E_j}{d_j^2},\chi_is(\chi_i)\rangle\frac{E_i}{d_i^2}=\sum_i \frac{\langle \gL,\chi_is(\chi_i)s(\chi_j)\rangle}{d_j}\frac{E_i}{d_i^2}
\end{eqnarray*}
where the last equality follows from \eqref{dual1}.

(ii) All entries are non-negative rational numbers since the $(i,j)$ entry of $A$ equals the number of times $\chi_j$ appears as a constituent of $\chi_is(\chi_i)$ divided by $d_j.$ 

(iii) Follows from the definition of $A$ since $\chi_0=\gep.$

(iv) The proof follows by induction on $m.$ For $m=1$ this follows from part (ii) since $A^0=\Id.$ Assume $(A^m)_{k0}=\sum_j(A^{m-1})_{kj}$  for all $0\le k\le n-1.$ Then:
$$(A^{m+1})_{i0}=(AA^m)_{i0}=\sum_k A_{ik}(A^{m})_{k0}= \sum_{k,j} A_{ik}(A^{m-1})_{kj}= \sum_j (A^m)_{ij},$$
where the third equality follows from the induction hypothesis.

The last part follows by induction since the first column of $A$ consists of $1$'s and all other entries are non-negative.   
\end{proof}
 We refer to $A$ as the {\bf Commutator matrix} of $H.$ Observe that this matrix  depends only on the Grothendieck ring of $H.$ 

\medskip 

We next define an important family of central iterated commutators which will play a key role in the sequel.

Define,
\begin{equation}\label{gammai}\gamma_0=\gL,\,\gamma_1=T(\gL)=\{\gL,\gL\},\,\dots,\,\gamma_m=T^m(\gL)=\{\gamma_{m-1},\gL\}.\end{equation} 

Note that $\gga_1$ is the Hopf analogue of the {\it extensive commutator} $z$ in $kG$ given by:
\begin{equation}\label{ext}z=\frac{1}{|G|^2}\sum_{a,b\in G} aba\minus b\minus\end{equation}

We show,
\begin{proposition}\label{induction}
Let $\gga_m$ be defined as in \eqref{gammai} and the matrix $A$ be defined as in Proposition \ref{matrixA}. Then $$\gamma_{m}=\sum_i (A^m)_{i0}\frac{E_i}{d_i^2}=\sum_i\left(\sum_j(A^{m-1})_{ij}\right)\frac{E_i}{d_i^2}$$ for all $m\ge 1.$ Moreover, the coefficient of each $E_i$ in $\gamma_m$ is a non-zero rational number, in particular $\gamma_m$ is invertible.  
\end{proposition}
\begin{proof}
Since $\gL=\frac{E_0}{d_0^2},$ it follows from Proposition \ref{matrixA}(i) that the coordinates of $T^m(\gL)$ with respect to the basis $\{\frac{E_i}{d_i^2}\}$ is the first column of $A^m.$ The result follows now from Proposition \ref{matrixA}(iv). 
\end{proof}

Note that in particular $\gamma_1=\sum_i\frac{1}{d_i^2}E_i$. This result was proved  also in \cite{cwcom}. 
We show now,

\begin{lemma}\label{le}
Let $\gga_m$ be defined as in \eqref{gammai}. Then:

\medskip\noin (i). For each irreducible character $\chi_i,\;\langle \gamma_m,\chi_i\rangle$ is a non-negative rational number satisfying $\langle \gamma_m,\chi_i\rangle\le d_i.$   In particular
$$\langle \gga_1,\chi_i\rangle =\frac{1}{d_i}.$$

\medskip\noin (ii). $\langle \gamma_{m+1},\chi_i\rangle=d_i$ if and only if $\langle \gamma_m,\chi_j\rangle=d_j$ for each irreducible constituent of $\chi_is(\chi_i).$
\end{lemma}
\begin{proof}
(i). Proposition \ref{induction} implies that $\langle \gamma_m,\chi_i\rangle$ is a non-negative rational number for all $m.$ 
Also,
For $m=0,\;\langle\gL,\chi_i\rangle=\gd_{i,0}\le d_i.$ For $m=1$ we have $\gamma_1=\sum_i\frac{1}{d_i}E_i^2$ hence $\langle \gamma_1,\chi_i\rangle=\frac{1}{d_i}\le d_i.$ 

Assume  by induction $\langle \gamma_m,\chi_i\rangle\le d_i$ for all $i.$ Then by \eqref{zi},
$$ \langle \gamma_{m+1},\chi_i\rangle=\frac{1}{d_i}\langle \gamma_m,\chi_is(\chi_i)\rangle$$
Let $\chi_is(\chi_i)=\sum m_j\chi_j,$ then $\sum m_jd_j=d_i^2$ and we have:
\begin{equation}\label{leq}\langle \gamma_{m+1},\chi_i\rangle=\frac{1}{d_i}\sum m_j\langle \gamma_m,\chi_j\rangle\le\frac{1}{d_i}\sum m_jd_j= d_i\end{equation}

(ii). If $\langle \gamma_{m+1},\chi_i\rangle=d_i$ then equality holds in \eqref{leq}. Since all $m_j$ are positive we must have $\langle \gamma_m,\chi_j\rangle=d_j$ for each irreducible constituent $\chi_j.$ Conversely, if $\langle \gamma_m,\chi_j\rangle=d_j$ for each irreducible constituent of $\chi_is(\chi_i)$ then equality holds in \eqref{leq}.
\end{proof} 
\medskip
Set 
\begin{equation}\label{s}S_m={\rm Sp}_k\{\chi_i\in {\rm Irr}(H)\,|\,\langle\gga_m,\chi_i\rangle=d_i\}\end{equation}
and
\begin{equation}\label{cs} HS_m=\text{the Hopf subalgebra of $H^*$ generated by }\,S_m.\end{equation}

We show,
\begin{lemma}\label{based} Let $S_m$ and $HS_m$ be defined as above. Then:

\medskip\noin (i) $S_0=k$ and $S_1=kG(H^*).$ 

\medskip If moreover $\chi_is(\chi_i)\in Z(R(H))$ for each irreducible character $\chi_i,$ then:

\medskip\noin (ii)  $S_m$ is a based ring for all $m.$  That is, if $\chi_i,\chi_j\in S_m$ then all the irreducible constituents of $\chi_i\chi_j$ are in $S_m$ as well.

\medskip\noin (iii) $S_m=\{\chi_i\,|\,\chi_is(\chi_i)\in S_{m-1}\}$ hence $(HS_m)_{ad}\subseteq HS_{m-1}.$ 
\end{lemma}
\begin{proof}
(i). Since $\gga_0=\gL$ we have $S_0=k.$ By Lemma \ref{le}.1, $\langle \gamma_1,\chi_i\rangle=\frac{1}{d_i},$ which equals $d_i$ if and only if $d_i=1,$ that is if and only if $\chi_i\in G(H^*).$ 

\medskip (ii). 
Observe that centrality of $\chi_js(\chi_j)$ implies that $$\chi_is(\chi_i)\chi_js(\chi_j)=\chi_i\chi_js(\chi_i\chi_j).$$

Clearly $S_1=kG(H^*)$ is  a based ring. Assume by induction that $S_{t-1}$ is a based ring, and let  $\chi_i,\chi_j\in S_t.$ By Lemma \ref{le} we have that $\langle \gga_{t-1},\chi_k\rangle=d_k$ for each irreducible constituent $\chi_k$ of $\chi_is(\chi_i)$ or $\chi_js(\chi_j),$    hence by the induction hypothesis,  we have  $\chi_is(\chi_i)\chi_js(\chi_j)\in S_{t-1}.$  Let $\chi_l$ be an irreducible constituent of $\chi_i\chi_j.$ Since all constituents of $\chi_ls(\chi_l)$ are constituents of $\chi_i\chi_js(\chi_i\chi_j)=\chi_is(\chi_i)\chi_js(\chi_j)\in S_{t-1},$ it follows that   $\chi_l\in S_t.$  
Hence $S_t$ is a based ring. 

\medskip (iii). Follows directly from part (ii) and Lemma \ref{le}(ii).
\end{proof}
\begin{rema} When $R(H)$ is commutative the series of based rings $S_m$ given in \eqref{s} coincides with the lower series defined in  \cite[4.12]{gn}. However the assumption $\chi_is(\chi_i)\in Z(R(H))$ is  weaker than the assumption of commutativity of $R(H).$   For example, if $H=(kG)^*,\;G$ a non-abelian finite group then $R(H)=kG$ is not commutative, yet $\chi_is(\chi_i)=s(\chi_i)\chi_i=1.$ Same is true for $H=D(kG)^*.$
\end{rema}

\medskip
We can show now the main result of this section.

\begin{theorem}\label{connection}
Let $H$ be a semisimple Hopf algebra over an algebraically closed field of characteristic $0,$ and let $N_t,\,\gga_t$ be defined as in \eqref{N} and \eqref{gammai}. Assume $\chi_is(\chi_i)\in Z(R(H))$ for each irreducible character $\chi_i.$ Then $H$ is nilpotent if and only if $\gga_m=1$ for some $m\in\Z^+.$  

Its index of nilpotency is  the least integer $m$ so that $\gga_m=1.$

\end{theorem}
\begin{proof}
The proof follows by induction. Let $\{B_t\}$ be the series defined in \eqref{b}. If $\gga_m=1$ then $S_m=R(H)$ and by Lemma \ref{based}.(iii) we have $(H^*)_{ad}=(HS_m)_{ad}\subseteq HS_{m-1},$ hence $B_1\subset HS_{m-1}.$   Assume by induction $B_t\subseteq HS_{m-t}.$ Then
$$B_{t+1}=(B_t)_{ad}\subseteq (HS_{m-t})_{ad}\subseteq HS_{m-t-1},$$
where the last inclusion follows from Lemma \ref{based}.(iii). We have in particular $B_m\subseteq HS_0=k,$ which implies by Lemma \ref{zn} that $Z_m=H.$ Thus $H$ is nilpotent.  Conversely, if $H$ is nilpotent of degree $m$ then $N_m=k$ implying in particular that $\gga_m=1.$
\end{proof}

As a corollary we obtain,
\begin{coro}\label{ntilde}
Let $H$ be a semisimple Hopf algebra over an algebraically closed field of characteristic $0$ and $\{N_t\}$ be defined as in \eqref{N}. Assume $\chi_is(\chi_i)\in Z(R(H))$ for each irreducible character $\chi_i.$ Then 
$$N_t=\text{The left coideal subalgebra generated by }\,\gga_t.$$
\end{coro}
\begin{proof}
Let $L_t$ denote the left coideal subalgebra generated by $\gga_t\leftharpoonup H^*.$ Then $L_t\subseteq N_t.$ Since $\gga_t$ is central it follows by \cite[Prop. 2.5]{cw1} that $L_t$ is also left normal. Let $\pi:H\rightarrow \ol{H}=H//L_t.$ Since ${\rm Irr}(\ol{H})\subset {\rm Irr}(H)$ the cetrality assumption holds for $\ol{H}$ as well. Then $\gga_t(\ol{H})=\pi(\gga_t)=1.$ By Theorem \ref{connection}, $N_t(\ol{H})=k.$ Since $\pi(\gL)=\ol{\gL}$ it follows by induction that $\pi(N_t)\subseteq N_t(\ol{H})=k.$ Since $N_t$ is a normal left coideal subalgebra and $L_t=H^{co\pi}$ we must have $N_t\subseteq L_t.$
\end{proof}

Another corollary relates nilpotency of $H$ and the eigenvalues of $A.$ 
\begin{theorem}\label{atothm}
Let $H$ be a semisimple Hopf algebra over an algebraically closed field of characteristic $0$ and assume $\chi_is(\chi_i)\in Z(R(H))$ for each irreducible character $\chi_i.$ Then $H$ is nilpotent  if and only if its commutator matrix $A$ has eigenvalues $\{1,0\}$ where the algebraic multiplicity of $1$ is $1.$
\end{theorem}
\begin{proof}
Theorem \ref{connection} and \eqref{tnm} imply that $H$ is nilpotent of degree $m$ if and only if the operator  $T$ satisfies $T^m=\gep_{|Z(H)}.$ That is, $T^m(E_i)=\langle \gep,E_i\rangle 1=\gd_{i,0}$ for each central idempotent $E_i.$   By Proposition \ref{matrixA} this is equivalent to:
$$A^{m}=
\left(\begin{array}{clrr}
1&0&\cdots&0\\
d_1^2&0&\cdots&0\\
\vdots&\vdots&\;&\vdots\\
d_{n-1}^2&0&\cdots&0	
\end{array}\right)
$$
It follows that the eigenvalues of $T^m$ are $1$ and $0.$   Hence the eigenvalues of $T$ are roots of unity and $0.$ Since $1$ is an eigenvector of $T$ with eigenvalue $1,$  and  the alegbraic multiplicity of $1$ in $T^m$ is $1,$ it follows that the only possible root of unity is $1$ and its alegbraic multiplicity is $1.$  
\end{proof}

\medskip Upon another assumption on $R(H),$ which holds in particular when $R(H)$ is commutative, yet it is a weaker assumption, we show:

\begin{proposition}\label{xsx}
Assume $\chi_is(\chi_i)=s(\chi_i)\chi_i$ for each irreducible character $\chi_i.$ Then:

\medskip\noin (i).  For all $m,\;\gamma_m=S\gamma_m$ 

\medskip\noin (ii). For all $m,\;\{\gL,\gga_m\}=\{\gga_m,\gL\}.$

\medskip\noin (iii).  $\gga_{m+1}\in \gga_{m}\leftharpoonup H^*$ for all $m\ge 0.$
\end{proposition}
\begin{proof}
(i). If $\chi_is(\chi_i)=s(\chi_i)\chi_i$ for each irreducible character $\chi_i$ then the matrix $A$ defined in \eqref{A} satisfies $A_{ij}=A_{i^*j}$ for all $i,j.$ Assume by induction $(A^m)_{ij}=(A^m)_{i*j},$ then
$$(A^{m+1})_{il}=\sum_j (A^m)_{i^*j}A_{jl}=(A^{m+1})_{i^*l}.$$
By Proposition \ref{induction} the expression above implies that the coefficient of $E_i$ in $\gamma_m$ equals the coefficient of $E_{i^*}$ in $\gamma_m,$ hence $\gamma_m=S\gamma_m.$ 

(ii).  Since $\gamma_m=S\gamma_m$ we have   $\sum {\gga_m}_1\ot {\gga_m}_1=\sum S{\gga_m}_2\ot S{\gga_m}_1.$ Let $\gga=\gga_m,$ Since $\gL\ad h\in Z(H)$ for all $h,$ we have,
$$\sum\gga_1\gL_1S\gga_2S\gL_2=\sum\gL_1S\gga_2S\gL_2\gga_1=\sum\gL_1\gga_1S\gL_2S\gga_2.$$

(iii). We have
$$\gamma_{m+1}=S\gamma_{m+1}=\sum \gL_2\gamma_{m_2}S\gL_1S\gamma_{m_1}=\sum(\gL\ad {\gamma_m}_2)S{\gamma_m}_1.$$
Since $\gga_{m}\leftharpoonup H^*$ is stable under the adjoint action of $H$ (by \cite{cw1}) and since ${\gga_m}_2, {S\gamma_m}_1\in \gga_m\leftharpoonup H^*,$ it follows that $\gga_{m+1}\in\gga_m\leftharpoonup H^*.$ 
\end{proof}

\section{Probabilistic methods for Hopf algebras with a commutative character algebra.}

Recall that $H$ is a Frobenius algebra. One defines a Frobenius map $\Psi
:H_{H^*}\rightarrow H^*_{H^*}$ by
\begin{equation}\label{psi}\Psi(h)=\gl\leftharpoonup S(h)\qquad \Psi\minus(p)=\gL\leftharpoonup p.\end{equation}
where $H^*$ is a right $H^*$-module under multiplication and $H$ is a right $H^*$-module under right {\it hit}. If $H$ is semisimple then
$$\Psi(Z(H))=R(H).$$

 \medskip  Let $\frac{1}{d}\gl=F_0,\dots, F_{n-1}$ be the set of central primitive idempotents of $R(H).$ When $R(H)$ is commutative then $\{F_j\}$   form another basis for $R(H).$
Define as in \cite{cwturki} the {\bf conjugacy class} ${\mathfrak C}_i$ as:
$${\mathfrak C}_i=\gL\leftharpoonup F_iH^*.$$

Note $\dim {\mathfrak C}_i=\dim(F_iH^*).$ We generalize also the notions of  {\bf Class sum}  and of a representative of a conjugacy class as follows:
\begin{equation}\label{ci} C_i=\gL\leftharpoonup dF_i=\Psi\minus(dF_i)\qquad \eta_i=\frac{C_i}{\dim {\mathfrak C}_i}.\end{equation}
We refer to $\eta_i$ as a  {\bf normalized class sum}. 
In \cite[(11)]{cwturki} we show that 
\begin{equation}\label{filamda}\langle F_i,\gL\rangle =\frac{\dim{\mathfrak C}_i}{d}.\end{equation}
Hence $\langle\gep,\eta_i\rangle =1$ for all $i$ and the normalized class sums $\{\eta_i\}$ form a basis for $Z(H)$ dual to  $\{F_i\}.$  
We can define now a generalized character table  for $H,$

$$(\xi_{ij})\quad\text{where }\;\xi_{ij}=\left\langle  \chi_i,\eta_j\right\rangle,$$
$0\le i,j\le n-1.$
Note that $\eta_0=1$ and so $\xi_{i0}=\langle  \chi_i,1\rangle  =d_i.$  Moreover, $(\xi_{ij})$ is the change of bases matrix between $\{\chi_i\}$ and $\{F_i\}.$

\bigskip

We call an element $z\in Z(H)$  a {\bf central distribution element} if 
$$z=\sum\ga_i\eta_i,\;\ga_i\in \R^+\cup\{0\},\quad \sum_i\ga_i=1.$$  Note that since $\langle\gep,\eta_i\rangle=1$ for all $i,$ it follows that  $\sum_i\ga_i=1$ is equivalent to $\langle\gep,z\rangle=1.$ 

The central distribution element $z$ defines a  distribution $X_z$ on $H$ by letting:
$${\rm Prob}(X_z=C_i)=\ga_i$$
then the corresponding distribution function $f_z$ is given by:
$$f_z(C_i)={\rm Prob}(X_z=C_i)=\ga_i.$$

Since $\{C_i\}$ form a basis for $R(H)$ it follows that the distribution function  $f_z$ defines an element  in $\hom(Z(H),\C)=R(H).$ By using dual bases this element is given precisely by 
\begin{equation}\label{fz}f_z=\sum_i f_z(C_i)\frac{F_i}{\dim{\mathfrak C}_i}=\sum\ga_i \frac{F_i}{\dim{\mathfrak C}_i}=\frac{1}{d}\Psi(z),\end{equation}
where $\Psi$ is the Frobenius  function given in \eqref{psi} and the last equality follows from \eqref{ci}.

\begin{example}\label{uniform}
(I) The idempotent integral $\gL$ is a central distribution element since by \eqref{filamda} and  dual bases, 
$$\gL=\sum_i\frac{\dim{\mathfrak C}_i}{d}\eta_i.$$ 
The corresponding  distribution $X_{\gL}$ is given by: $${\rm Prob}(X_{\gL}=C_i)=\frac{\dim{\mathfrak C}_i}{d}.$$ (For groups this is called the uniform distribution on the group since it satisfies Prob$(X_{\gL}=g)=\frac{1}{|G|}$ for all $g\in G$). 

\medskip\noin(II) Another central distribution element is the normalized class sum $\eta_i.$ It defines the distribution $X_i$, where
${\rm Prob}(X_i=C_j)=\gd_{ij}.$ 

\medskip\noin (III) The character of the adjoint representation is related to the the following central distribution element. Let  $z_{ad}=\frac{1}{n}\sum\eta_i.$ Then it  defines a uniform distribution $U_{ad}$ on class sums. That is, 
$${\rm Prob}(U_{ad}=C_i)=\frac{1}{n}.$$
The corresponding distribution function is given by:
$$f_{z_{ad}}=\frac{1}{nd}\chi_{ad}.$$ This follows from \eqref{ci} since $\chi_{ad}=\sum\frac{d}{\dim{\mathfrak C}_i}F_i$ (see e.g. \cite[Th.2.2]{cwturki}).
\end{example}

Starting from  distribution functions, we have:
\begin{lemma}\label{distfunction}
Assume $R(H)$ is commutative. An element $f\in R(H)$ is a distribution function on $Z(H)$ if and only if 
$$ (i) \;\langle f,\eta_i\rangle\in\R^+\cup \{0\}\;\forall i\quad\text{and}\quad (ii)\; \langle f,\gL\rangle=\frac{1}{d}.$$ 
In this case $z=d\Psi\minus(f)$ is a central distribution element and $$f_z=f=\frac{1}{d}\sum_i \langle\chi_i,z\rangle s(\chi_i).$$   
\end{lemma}
\begin{proof}
By dual bases,  $f=\sum \langle f,\eta_i\rangle F_i,$ hence  by \eqref{ci}, 
\begin{equation}\label{zf}\Psi\minus(f)= \frac{1}{d}\sum\langle f,\eta_i\rangle (\dim{\mathfrak C}_i)\eta_i.\end{equation}
Let 
$z=d\Psi\minus(f)=\sum\langle f,\eta_i\rangle (\dim{\mathfrak C}_i)\eta_i,$
then  $f=f_z.$ We claim that $z$ is a central distribution element. Indeed, by assumption (i) on $f$ all coefficients of $\{\eta_i\}$'s are non-negative real numbers. Now, 
$$\langle\gep,z\rangle=\langle\gep,df\rightharpoonup\gL\rangle=d\langle f,\gL\rangle=1,$$
where the last equality follows from assumption (ii) on $f.$

By the other set of dual bases, $z=\sum_i \langle z,\chi_i\rangle\frac{1}{d_i}E_i.$ The last part follows from \eqref{dual1}, since $f_z=\frac{1}{d}\Psi(z).$ 
\end{proof}

\begin{example}\label{frob}
Counting functions for groups give rise to  probability functions on the group algebras after dividing by an appropriate power of $|G|.$ Below are two examples:

(I) Frobenius proved  that the function on a finite group $G$ that counts the number of ways an element of $G$ can be realized as a commutator  is given by:
$$f_{rob}=\sum_i \frac{|G|}{d_i}\chi_i.$$
Since the total number of commutators $aba\minus b\minus,\;a.b\in G,$ is $G\times G,$ we make the counting function $f_{rob}$ into  a distribution function by taking:
$$f=\frac{1}{|G|^2}f_{rob}.$$

Consider $f_{rob}$ as an element of $R(H)$ where $H=kG.$ We claim that $f$ is a distribution function in the sense of Lemma \ref{distfunction}. Indeed, since  $\eta_i=\frac{C_i}{|{\mathfrak C}_i|},$ we have for arbitrary  $g_i\in{\mathfrak C}_i,$ 
$$\langle f_{rob},\eta_i\rangle=\langle f_{rob},g_i\rangle\ge 0.$$
Hence property (i) is satisfied. Since $\langle f_{rob},\gL\rangle=|G|\langle\gep,\gL\rangle=|G|,$ property (ii) follows. Thus $f$ is a distribution function.

By \eqref{zf}   the central distribution element corresponding to $f$ is given by:  
$$z=\frac{1}{|G|^2} \sum_i\langle f_{rob},\eta_i\rangle|{\mathfrak C}_i|\eta_i,$$

An explicit realization of $z$ in terms of the elements of the group is the {\it extensive commutator} $z$ as in \eqref{ext}, 
$$z=\sum\frac{1}{|G|^2}\sum_{g\in G}aba\minus b\minus.$$
 Indeed, choose arbitrarily $g_i\in{\mathfrak C}_i,$ then,
\begin{eqnarray*}\lefteqn{\frac{1}{|G|^2}\sum_{a,b\in G}aba\minus b\minus=}\\
&=&\frac{1}{|G|^2}\sum_{g\in G}\langle f_{rob},g\rangle g=\frac{1}{|G|^2} \sum_i\langle f_{rob},g_i\rangle C_i=\frac{1}{|G|^2} \sum_i\langle f_{rob},\eta_i\rangle|{\mathfrak C}_i|\eta_i=z.\end{eqnarray*} 

In \cite{cwcom} the function $f_{rob}=\sum_i \frac{d}{d_i}\chi_i$ was generalized to any semisimple Hopf algebra. We showed there that:   
$$\frac{1}{d^2}f_{rob}=\frac{1}{d}\Psi(\gga_1).$$

In what follows we show that the Hopf analogue $\gga_1$ of the {\it extensive commutator} $z$ is  a central distribution element.  Analogously, it follows from \eqref{fz} that $\frac{1}{d^2}f_{rob}$ is the distribution function related to $\gga_1.$ 

\medskip (II) Another counting function for groups is the root function. It counts the number of solutions in $G$ to the equation $x^m=g,\,g\in G.$ This function is given by:
$$N_{mrt}=\sum_i\langle \chi_i,\frac{1}{|G|}\sum_{g\in G}g^m\rangle\chi_i.$$
The coefficient of $\chi_i$ is called the $m$-th Frobenius-Schur indicator.
The corresponding distribution function  is obtained by dividing by $|G|,$ that is:
$$f_{mrt}=\frac{1}{|G|^2}\sum_i\langle \chi_i,\sum_{g\in G}g^m\rangle\chi_i.$$
It is straightforward to check that $f_{mrt}$ is a ditribution function for the group algebra. The corresponding central distribution element is $\frac{1}{|G|}\sum_{g\in G}g^m.$

In \cite{lm} the Frobenius-Schur indicator was generalized to any semisimple Hopf algebras  where the element $\frac{1}{|G|}\sum_{g\in G}g^m$ is  generalized to the central element $\gL^{[m]}=\sum\gL_1\cdots\gL_m.$ By Lemma \ref{distfunction},
$$f_{mrt}=\frac{1}{d}\sum_i\langle \chi_i,\gL^{[m]}\rangle\chi_i.$$
We do not know if in general $f_{mrt}$ is a distribution function. We  ask:

\medskip\noin{\bf Question:} Let $H$ be a semisimple Hopf algebra so that $R(H)$ is commutative. When is it true that $\langle f_{mrt},\eta_i\rangle\in\R^+\cup\{0\},$ (or equivalently is $\langle F_i,\gL^{[m]}\rangle\in\R^+\cup\{0\}$) for all $i?$
\end{example}

\medskip
In what follows we investigate {\it probabilistically nilpotet} Hopf algebras, a notion that holds trivially for finite groups. We show first,
\begin{lemma}\label{etaseta}
For each normalized class sum $\eta_i$ we have $$\{\eta_i,\gL\}=\eta_iS\eta_i.$$ 
\end{lemma}
\begin{proof}
Since ${\mathfrak C}_i$ is a left coideal and $\eta_i\in {\mathfrak C}_i,$ it follows that each component ${\eta_i}_2\in {\mathfrak C}_i.$ By \cite[Prop.3.6]{cwturki} we have that $\gL\ad{\eta_i}_2=\langle\gep,{\eta_i}_2\rangle{\eta_i}.$ Hence
$$\{\eta_i,\gL\}=\sum{\eta_i}_1\gL\ad S{\eta_i}_2=\sum{\eta_i}_1S(\gL\ad{\eta_i}_2)=\eta_iS\eta_i.$$
\end{proof}
Next we show,
\begin{lemma}\label{product}
If $H$ is quasitriangular then $$\eta_i\eta_j=\sum q_{ijk}\eta_k,\quad q_{ijk}\in\Q^+\cup\{0\}.$$  Hence the product of two central distribution elements is a central distribution element.
\end{lemma} 
\begin{proof}
The first part follows from  \cite[Th.2.6]{cwmia}.  The second part  follows from the first part and the fact that $\gep$ is multiplicative. 
\end{proof}
The lemmas above imply,
\begin{proposition}\label{quasieta}
Assume $H$ is quasitriangulr then the following hold:

\medskip\noin (i) For any central distribution element $z,\;T(z)=\{z,\gL\}$ is a central distribution element as well.

\medskip\noin (ii) The matrix of the opertator $T$ with respect to the basis $\{\eta_i\}$ is a non-negative rational matrix which equals $PAP\minus$ where  $A$ is as defined in Proposition \ref{matrixA}, $(\xi_{ij})$ is the character table of $H$ and 
$$P=\left({\rm diag}\{\dim{\mathfrak C}_0,\dots,\dim{\mathfrak C}_{n-1}\}\right)(\xi_{ji^*})\left({\rm diag}\{d\minus_0,\dots,d\minus_{n-1}\}\right)$$ 
\end{proposition}
\begin{proof}
(i)  If $z=\sum\ga_i\eta_i,\,\ga_i\ge 0,$ then by Lemma \ref{etaseta} $\{z,\gL\}=\sum_i\ga_i\eta_iS\eta_i$ which is a central distribution element since all the coefficients are non-negative by Lemma \ref{product} and $\langle\gep,  \{z,\gL\}\rangle=1.$

\medskip(ii) The fact that the matrix of $T$ with respect to the basis $\{\eta_i\}$ has non-negative rational entries follows from Lemma \ref{etaseta} and Lemma \ref{product}. 

By Proposition \ref{matrixA}(i), the matrix $A$ is the matrix of $T$ with respect to the basis $\{\frac{E_i}{d_i^2}\}.$ 
Now, the character table  is the change of bases matrix between $\{F_i\}$ and $\{\chi_i\}$ in $R(H)$ (see e.g. \cite[Th.3.1]{cwturki}. Applying $\Psi\minus$ yields that:
$$\frac{E_j}{d_j^2}=\frac{1}{dd_j}\sum_i\xi_{ji^*}(\dim{\mathfrak C}_i)\eta_i.$$
These imply the desired result.

\end{proof}

We can show now an essential property of the iterated commutator:
\begin{proposition}\label{converge}
Let $z\in Z(H)$ be a central distribution element and let $T(z)=\{z,\gL\}.$ Then $${\rm Prob}(T^m(z)=1)\longrightarrow 1\quad\text{as}\quad m\rightarrow \infty.$$
\end{proposition}
\begin{proof}

By \eqref{ci}, $\Psi(\eta_i)=\frac{d}{\dim {\mathfrak C}_i}F_i,$ hence 
$$\langle\eta_iS(\eta_i),\gl\rangle=\langle\eta_i,\Psi(\eta_i)\rangle=\frac{d}{\dim {\mathfrak C}_i}\langle\eta_i,F_i\rangle=\frac{d}{\dim {\mathfrak C}_i}.$$
On  the other hand, the dual bases imply that $\langle\eta_i,\gl\rangle=d\gd_{i,0}.$ Hence if $\eta_iS\eta_i=\gb_i\cdot 1+\sum_{j\ne 0}\gb_{j_i}\eta_j$ then 
$\langle\eta_iS\eta_i,\gl\rangle=d\gb_i.$ Hence $\gb_i=\frac{1}{\dim {\mathfrak C}_i}$ and we have:
\begin{equation}\label{bsb} \eta_iS\eta_i=\frac{1}{\dim {\mathfrak C}_i}\cdot 1 +\cdots\end{equation}

Assume $z=\ga_0\cdot 1+\sum_{i\ne 0}\ga_i\eta_i.$  By Lemma \ref{etaseta} we have 
$$T(z)=\ga_0+\sum_{i\ne 0}\ga_i\eta_iS\eta_i,$$ hence by \eqref{bsb}
\begin{equation}\label{tz}T(z)=\left(\ga_0+\sum_{i\ne 0}\frac{\ga_i}{\dim {\mathfrak C}_i}\right)\cdot 1+\sum_{i\ne 0}\ga'_i\eta_i.\end{equation}
Let $c=\max\{\dim {\mathfrak C}_i\}.$  Since $\langle\gep,T(z)\rangle=1$ it follows that 
\begin{eqnarray*}\lefteqn{\sum_{i\ne 0}\ga'_i=
1-\ga_0-\sum_{i\ne 0} \frac{\ga_i}{\dim {\mathfrak C}_i}}\\
&\le& 1-\ga_0-\sum_{i\ne 0} \frac{\ga_i}{c}\\
&=&1-\ga_0- \frac{1}{c}(1-\ga_0)\qquad(\text{since }\,\langle\gep,z\rangle=\sum_i\ga_i=1)\\
&=&(1-\ga_0)(1-\frac{1}{c}).\end{eqnarray*}
That is,
\begin{equation}\label{tzr}
\sum_{i\ne 0}\ga'_i\le (1-\ga_0)(1-\frac{1}{c}).\end{equation} 
Assume by induction
$$ T^m(z)=\ga_m\cdot 1+\sum_{i\ne 0}\ga_{mi}\eta_i\quad\text{where}\quad \sum_{i\ne 0}\ga_{mi}\le (1-\ga_0)(1-\frac{1}{c})^m$$
Then by \eqref{tz}, 
$$T^{m+1}(z)= \left(\ga_m+\sum_{i\ne 0}\frac{\ga_{mi}}{\dim {\mathfrak C}_i}\right)\cdot 1+\sum_{i\ne 0}\ga_{m+1,i}\eta_i.$$
By \eqref{tzr},
\begin{eqnarray*}\lefteqn{\sum_{i\ne 0}\ga_{m+1,i}\le (1-\ga_m)(1-c)=}\\
&=&\sum_{i\ne 0}\ga_{mi}(1-c)\quad(\text{since }\langle\gep,T^m(z)\rangle=1)\\
&\le&(1-\ga_0)(1-\frac{1}{c})^m(1-c)\quad(\text{by induction hypothesis})\\
&=& (1-\ga_0)(1-\frac{1}{c})^{m+1}
\end{eqnarray*}
It follows that $\sum_{i\ne 0}\ga_{m,i}\rightarrow 0.$  and so $\ga_m\rightarrow 1$ as $m\rightarrow \infty.$
\end{proof}

Note that by \eqref{gammai},  If $H$ is quasitriangular then $\gga_m$ is a central distribution element for all $m\ge 0.$ 
By Theorem \ref{connection} $H$ is nilpotent if and only if $\gga_m=1.$ 
This motivates the following definition. 

\begin{definition}\label{pnil} A semisimple Hopf algebra is {\bf Probabilistically nilpotent} if 
$${\rm Prob}(\gga_m=1)\longrightarrow 1\quad\text{as}\quad m\rightarrow \infty.$$
\end{definition}
Proposition \ref{converge} yields now the main result of this section:
\begin{theorem}\label{main3}
Let  $H$ be a semisimple quasitriangular Hopf algebra over an algebraically closed field of characteristic $0.$ Then $H$ is probabilistically nilpotent. 
\end{theorem}

\medskip
The following result is a generalization of \cite[Th.4.10]{av}. It is an algebraic statement about the eigenvalues of the commutator matrix $A.$

\begin{theorem}\label{eigen}
Let $H$ be a semisimple quasitriangular Hopf algebra over $\C.$ Then the commutator matrix $A$  has $1$ as an eigenvalue with corresponding $1$-dimensional eigenspace. 
 All other eigenvalues $c$ satisfy $|c|<1.$
\end{theorem}
\begin{proof}

Since $A$ is the matrix of the operator $T(z)=\{z,\gL\}$ we show the statement for $T.$

Observe first that if $0\ne z=\sum\ga_i\eta_i$ so that $\ga_i\in\R,\;\ga_i\ge 0$ for all $i,$ then necessarily $\langle\gep,z\rangle>0$ and $\tl{z}=\frac{1}{\langle\gep,z\rangle}z$ is a central distribution.   

By Proposition \ref{converge}, ${\rm Prob}(T^{m}(\tl{z})=1)\longrightarrow 1.$ In this case we say that $T^m(\tl{z})$ converges to $1$ and $T^m(z)$ converges to $\langle\gep,z\rangle \cdot 1.$ 

If $z$ is a real combination of $\{\eta_i\}$ then $z=z_+-z_-$ where $z_+$ and $z_-$ have only non-negative coefficients. By above $T^m(z)$ converges to  $\langle\gep,z_+\rangle -\langle\gep,z_-\rangle =\langle\gep,z\rangle\cdot 1.$

Assume now $v=\sum\ga_j\eta_j$ is an eigenvector for $T$ with eigenvalue $c.$  Since the matrix of $T$ with respect to the basis $\{\eta_i\}$ is real, it follows that  $\ol{v}=\sum\ol{\ga_i}\eta_i$ is an eigenvector with eigenvalue $\ol{c}.$ 

Assume $v\notin k.$ 
Then $\ga_j\ne 0$ for some $j>0.$ Since multiplying $v$ by  $\ga_j\minus$ yields another eigenvector, we may assume without loss of generality  that   $\ga_j=1.$ Now, $v+\ol{v}$ is a real vector hence we have  that 
$$T^m(v+\ol{v})=c^mv+\ol{c^mv}$$ converges to $(\langle\gep,v\rangle+\ol{\langle\gep,v\rangle})\cdot 1.$ In particular, taking the coefficient of $\eta_j$ we obtain,
$$2\Re(c^m)=c^m+\ol{c^m}\rightarrow 0\quad\text{as}\;m\rightarrow \infty.$$
 Similarly, $v-\ol{v}=\imath w$ where $w$ is a real vector. Hence 
$T^m(w)$ converges to $\langle\gep,w\rangle\cdot 1= \imath(\langle\gep,v\rangle-\ol{\langle\gep,v\rangle})\cdot 1.$
Since $T^m(w)=\imath(c^mv-\ol{c^mv})$ we can take again the coefficient of $\eta_j$ to obtain
$$-2\Im(c^m)=\imath(c^m-\ol{c^m})\rightarrow 0 \quad\text{as}\;m\rightarrow \infty.$$
Hence 
$$|c^m|^2= (\Re(c^m))^2+(\Im(c^m))^2\longrightarrow 0 \quad\text{as}\;m\rightarrow \infty,$$
implying that $|c|<1.$ 

Since  $T(1)=1,$ we have $1$ is an eigenvalue.  Its coordinate vector with respect to the basis $\{\frac{E_i}{d_i^2}\}$ is given by:
 
$$\left(\begin{array}{clrr}
1\\
d_1^2\\
\vdots\\
d_{n-1}^2	
\end{array}\right)
.$$.  
\end{proof}

\bigskip

\noin{\bf Acknowledgment}. We wish to thank Uzi Vishne for his patient clarifications of the probabilistic methods for groups used in his paper \cite{av}.

\end{document}